\documentclass{amsart}

\usepackage{xcolor}
\usepackage{fancyhdr}
\usepackage{amssymb,amsfonts,amsmath,amsthm,tikz,tikz-cd,graphicx,url,mathtools}
\usepackage[all,arc]{xy}
\usepackage{enumerate}
\usepackage{mathrsfs}
\usepackage[colorlinks=true]{hyperref}
\usepackage{mathtools}
\usetikzlibrary{decorations.pathmorphing,snakes}
\usetikzlibrary{knots}

\usepackage[alphabetic]{amsrefs}
\usepackage{comment}

\usepackage[T1]{fontenc}
\usepackage[sc,osf]{mathpazo}

\newtheorem{thm}{Theorem}[section]
\newtheorem{cor}[thm]{Corollary}
\newtheorem{prop}[thm]{Proposition}
\newtheorem{lem}[thm]{Lemma}
\newtheorem{conj}[thm]{Conjecture}

\theoremstyle{definition}
\newtheorem{defn}[thm]{Definition}

\theoremstyle{remark}
\newtheorem{rem}[thm]{Remark}

\newcommand{\Z}{\mathbb{Z}}
\newcommand{\RP}{\mathbb{RP}}
\newcommand{\ZZ}{\mathbb{Z}/2}

\newcommand{\RR}{\mathbb{R}}
\newcommand{\CC}{\mathbb{C}}

\newcommand{\cab}{\mathcal{B}}

\newcommand{\car}{\mathcal{R}}
\newcommand{\cau}{\mathcal{U}}

\newcommand{\nm}[1]{\left\|#1\right\|}

\DeclareMathOperator{\dimm}{dim}

\DeclareMathOperator{\rankk}{rank}

\DeclareMathOperator{\idd}{id}

\DeclareMathOperator{\SO}{SO}
\DeclareMathOperator{\SU}{SU}

\DeclareMathOperator{\KHI}{KHI}
\DeclareMathOperator{\Kh}{Kh}
\DeclareMathOperator{\Is}{I^{\sharp}}
\DeclareMathOperator{\In}{I^{\natural}}

\newcommand{\Khr}{\widetilde{\Kh}}

\usepackage{hyperref}

\title{Instantons and Khovanov homology in $\RP^3$}

\author{Hongjian Yang}
\email{yhj@stanford.edu}
\address{Department of Mathematics, Stanford University, Stanford, CA 94305}

\begin{document}

\begin{abstract}

We study the instanton Floer homology for links in $\RP^3$ and prove that the second page of Kronheimer--Mrowka's spectral sequence is isomorphic to the Khovanov homology of the mirror link. As an application, we prove that Khovanov homology detects the unknot and the standard $\RP^1$ in $\RP^3$. 

\end{abstract}

\maketitle

\section{Introduction}

One strategy to understand Khovanov homology \cite{Khovanov1999ACO} is to investigate its detection properties, which have been extensively studied for knots and links in $S^3$ \cite{kronheimer2011khovanov,baldwin2019khovanov,baldwin2022khovanov,baldwin2021khovanov,xie2019classification,martin2022khovanov,lipshitz2022khovanov,baldwin2021khovanov2}. Parallelly, it remains wildly open to generalize Khovanov homology to knots and links in other $3$-manifolds. This has been done for thickened surfaces \cite{asaeda2004categorification}, $S^1\times S^2$ \cite{rozansky2010categorification}, and the connected sum $\#^r(S^1\times S^2)$ \cite{willis2021khovanov}. Detection properties for these generalizations, especially for the case of thickened surfaces, have been explored in \cite{xie2021instantons,xie2019instanton,xie2020instantons,li2023instanton}. More recently, there has been growing interest in studying Khovanov homology for links in $\RP^3$ \cites{Gabrovek2013THECO,chen2021khovanov,manolescu2023rasmussen,chen2023bar}, probably one of the simplest closed $3$-manifolds aside from $S^3$. This paper is aimed at studying detection properties of Khovanov homology in $\RP^3$. 

Similarly to the original version, Khovanov homology for links in $\RP^3$ assigns a bigraded abelian group $\Kh(K)$ for each link $K\subset \RP^3$. The idea of construction is to view $\RP^3\backslash\text{pt}$ as a twisted $I$-bundle over $\RP^2$, and then project $K$ to $\RP^2$ to build a Khovanov-type chain complex from the cube of resolutions of the link diagram. There is also a reduced version $\Khr(K)$ for link $K$ with a marked point.
  
  There are two homology classes for links in $\RP^3$. Following \cite{manolescu2023rasmussen}, we say a link $L$ has \textit{class $0$} if it is null-homologous, and it has \textit{class $1$} if its first homology class $[L]\in H_1(\RP^3)=\ZZ$ is nontrivial. A link is said to be \textit{local} if it is contained in a $3$-ball in $\RP^3$. There are two knots deserve the name of unknot: the (local) unknot $U_1\subset B^3\subset \RP^3$ in class $0$, and the \textit{projective unknot} $U'$ in class $1$. The latter is defined as the standard embedded $\RP^1$ in $\RP^3$. Our main result is the following.

\begin{thm}\label{main theorem}
	Let $K$ be a knot in $\RP^3$. The reduced Khovanov homology $\Khr(K;\ZZ)$ has dimension $1$ if and only if $K$ is the unknot (when $K$ is of class $0$), or the projective unknot (when $K$ is of class $1$).
\end{thm}

Our proof of Theorem \ref{main theorem} follows Kronheimer and Mrowka's approach \cite{kronheimer2011khovanov}. The first step is to establish a spectral sequence relating Khovanov homology and singular instanton Floer homology for links in $\RP^3$. The spectral sequence behaves slightly different for two classes of knots, and we state the results in two sequent propositions.

\begin{thm}\label{ss class 0}
	Let $K$ be a class $0$ link in $\RP^3$. There is a spectral sequence with the $E_2$ page isomorphic to two copies of $\Kh(m(K);\ZZ)$, the Khovanov homology of the mirror of $K$, and it collapses to the singular instanton Floer homology $\Is(\RP^3,K;\ZZ)$. Similarly, for the reduced version, there is a spectral sequence with the $E_2$ page isomorphic to two copies of $\Khr(m(K);\ZZ)$ and abutting to the (reduced) singular instanton Floer homology $\In(\RP^3,K;\ZZ)$.
\end{thm}

For class $1$ links, the $E_2$ page is a variant of the original Khovanov homology.

\begin{thm}\label{ss class 1}
	Let $K$ be a class $1$ link in $\RP^3$. There is a spectral sequence with the $E_2$ page isomorphic to two copies of $\Kh_1(m(K);\ZZ)$, and it collapses to the singular instanton Floer homology $\Is(\RP^3,K;\ZZ)$. A similar result holds for the reduced version. Here, $\Kh_1$ is a deformed version of Khovanov homology, defined in Definition \ref{def Kh1}.
\end{thm}

\begin{rem}
	Previously, Chen \cite{chen2021khovanov} constructed a spectral sequence relating another variant of Khovanov homology and the Heegaard Floer homology of a branched double cover for null-homologous links in $\RP^3$, as an analogue of \cite{ozsvath2005heegaard}. It is interesting to ask whether Chen's variant also detects the unknot.
\end{rem}

\begin{cor}\label{cor rank ineq}
	Let $K$ be a link in $\RP^3$. Then we have
	\begin{equation*}
		2\dim\Kh(K;\ZZ)\ge\dim\Is(\RP^3,K;\CC).
	\end{equation*}
	Similar inequality holds for the reduced version: \begin{equation}\label{rank ineq}
		2\dim\Khr(K;\ZZ)\ge\dim\In(\RP^3,K;\CC).
	\end{equation}
\end{cor}

Now assume $K$ is a knot. Given the rank inequality (\ref{rank ineq}) and the isomorphism\[ \In(\RP^3,K;\CC)\cong \KHI(\RP^3,K;\CC)\]from \cite[Proposition 1.4]{kronheimer2011khovanov}, we deduce that \begin{equation}\label{eqn rank ineq KHI}
	2\dim\Khr(K;\ZZ)\ge\dim\KHI(\RP^3,K;\CC).
\end{equation} 


The second step is to classify all knots $K$ in $\RP^3$ satisfying \[\dim\KHI(\RP^3,K;\CC)=2.\] In general, it is related to a Floer-theoretic program to approach the Berge Conjecture \cite{hedden2011floer}. For our case, it essentially follows from \cite{li2023enhanced} and \cite{baldwin2021small}, using structural results in (sutured) instanton knot homology.

\subsection*{Organization of the paper}In Section \ref{sec bg}, we provide some background on Khovanov homology in $\RP^3$ and instanton Floer homology. We then study the instanton Floer homology of crossingless links in $\RP^3$ in Section \ref{sec Is unlink}. In Section \ref{sec E_1 differential}, we compute the differentials on the $E_1$ page of Kronheimer--Mrowka's spectral sequence. We then prove Theorems \ref{ss class 0} and \ref{ss class 1}, and Corollary \ref{cor rank ineq}. Finally, we prove Theorem \ref{main theorem} and include some further discussion in Section \ref{sec detection}.

\subsection*{Acknowledgments}The author would like to thank Ciprian Manolescu, Zhenkun Li, Yi Xie, and Fan Ye for helpful discussions. The author is also grateful to the anonymous referee for their valuable comments that helped clarify a mathematical issue and improve the writing quality of the manuscript. 

\section{Background}\label{sec bg}

\subsection{Khovanov homology for links in $\RP^3$}

Khovanov homology assigns a bigraded abelian group $\Kh(K)$ for link $K$ in $\RP^3$. It is originally defined in \cite{asaeda2004categorification} in $\ZZ$ coefficients, and is lifted to $\Z$ coefficients in \cite{Gabrovek2013THECO}. In \cite{manolescu2023rasmussen}, Manolescu and Willis defined a deformed Khovanov complex over the ring $\Z[s,t]$ with two variables and use it to define a Rasmussen invariant. We briefly review their construction in this section.

Let $K$ be a link in $\RP^3$. By identifying $\RP^3\backslash\text{pt}$ with the twisted $I$-bundle over $\RP^2$, we can project $K$ onto this $\RP^2$ to obtain a link diagram $D$ for $K$ with $N$ crossings. In practice, we often represent this by a diagram in a disk and keep in mind that antipodal points on the boundary are identified. For a crossing in $D$, we have two ways to resolve it: the $0$-smoothing and the $1$-smoothing, as showed in Figure \ref{01smoothing}. 

\begin{figure}[hbtp]
	\centering
	\begin{tikzpicture}
		
		\begin{knot}[ignore endpoint intersections=false,clip width=5,clip radius=5pt,looseness=1.2]
			
			\strand[thick] (-0.5,0.5)--(0.5,-0.5);
			\strand[thick] (0.5,0.5)--(-0.5,-0.5);
			
		\end{knot}
		\node[below] at (0,-0.7) {a crossing};
		\draw[->,snake=snake] (-0.8,0)--(-1.6,0);
		\draw[->,snake=snake] (0.8,0)--(1.6,0);
		
		\draw[thick] (-3,0.5)  to [out=315,in=180]  (-2.5,0.2) to [out=0,in=225]   (-2,0.5);
		\draw[thick] (-3,-0.5)  to [out=45,in=180]  (-2.5,-0.2) to [out=0,in=135]   (-2,-0.5);  \node[below] at (-2.5,-0.7) {$0$-smoothing};
		
		\draw[thick] (2,0.5)  to [out=315,in=90]  (2.2,0) to [out=270,in=45]    (2,-0.5);  \node[below] at (2.5,-0.7) {$1$-smoothing};
		\draw[thick] (3,0.5)  to [out=225,in=90]  (2.8,0) to [out=270,in=135]   (3,-0.5);
	\end{tikzpicture}
	\caption{Two types of smoothings.}\label{01smoothing}
\end{figure}

Fix an order of crossings, and the $2^N$ possible smoothings of $D$ can be indexed by $\{0,1\}^N$. For $u=(u_1,u_2,\dots,u_N),\,v=(v_1,v_2,\dots,v_N)$, we say $v\ge u$ if $v_i\ge u_i$ for all $i$. This gives a partial order on $\{0,1\}^N$. Let $D_v$ be the link diagram for the resolution $v\in\{0,1\}^N$. Two types of circles in $\RP^2$ may appear in $D_v$: the trivial one and the homological essential one, corresponding to the unknot and the projective unknot respectively.

\begin{lem}[{\cite[Lemma 2.2]{manolescu2023rasmussen}}]\label{lem reso circles}
	A homological essential circle appears in a resolution $D_v$ if and only if $K$ has class $1$. In this case, there is exactly one homological essential circle.
\end{lem}

The set $\{0,1\}^N$ can be equipped with the $\ell^1$ norm:\[\nm{v}_1=\sum_{i=1}^N|v_i|.\]For $u,v\in\{0,1\}^N$ such that $v\ge u$ and $\nm{v-u}_1=1$ (in this case, we say there is an \textit{edge} $e\colon u\to v$), there is an edge cobordism \[D_u\to D_v\]given by passing a saddle. It may merge two circles into one, or split one circle into two, or turn a circle into another circle. The first two cases have already appeared in the definition of (annular) Khovanov homology, and the last case, a $1$-to-$1$ bifurcation, does not appear for links in $S^3$ or thickened annulus. It can happen only when the link has class $0$. See Figure \ref{fig 1-1bifur} for clarification.

\begin{figure}[hbtp]
	\centering
	\begin{tikzpicture}
		\draw[->,snake=snake] (-0.5,0)--(0.5,0);
		
		\draw[thick] (-2,0.5)  to [out=315,in=180]  (-1.5,0.2) to [out=0,in=225]   (-1,0.5);
		\draw[thick] (-2,-0.5)  to [out=45,in=180]  (-1.5,-0.2) to [out=0,in=135]   (-1,-0.5);  
		
		\draw[thick] (-1.5,0) circle [radius=sqrt(2)/2];
		
		\draw[thick] (1,0.5)  to [out=315,in=90]  (1.2,0) to [out=270,in=45]    (1,-0.5);  
		\draw[thick] (2,0.5)  to [out=225,in=90]  (1.8,0) to [out=270,in=135]   (2,-0.5);
		
		\draw[thick] (1.5,0) circle [radius=sqrt(2)/2];
	\end{tikzpicture}
	\caption{A $1$-to-$1$ bifurcation. Notice that there is only one component on both sides as we are working on $\RP^2$, which is presented as a disk with antipodal points on the boundary attached. This can only happen when the link is in class $0$. }\label{fig 1-1bifur}
\end{figure}

For each $v\in\{0,1\}^N$, we fix an ordering for the circles in $D_v$ and an orientation for each circle in $D_v$. The collection of this data is called a set of \textit{auxiliary diagram choices} for $D$ in \cite{manolescu2023rasmussen}. Let $R$ be the ring $\Z[s,t]$. Let $V=\langle 1,X\rangle$ and $\overline{V}=\langle \overline{1},\overline{X}\rangle$ be two free rank $2$ graded $R$-modules. Roughly speaking, the \textit{deformed Khovanov complex} $KC_d^*(D)$ is defined as follows. First, to each oriented, ordered resolution $D_v$, we assign a triply graded chain group $C(D_v)$ as the tensor product of factors of $V$ for each trivial circle, and $\overline{V}$ for the homological essential circle (if present). Second, we define a differential $\partial^e$ for each edge $e\colon u\to v$ (see \cite[Table 6]{manolescu2023rasmussen}). It is identically zero for $1$-to-$1$ bifurcations and decomposes into homogeneous parts for merging and splitting maps:\[\partial^e=\partial^e_0+s\partial^e_-+st\Phi^e_0+t\Phi^e_+.\]These are enough to define the Khovanov chain complex in $\ZZ$ coefficients by taking the total complex as an iterated mapping cone. To lift it to $\Z$ coefficients, we adopt certain rules (see \cite[Section 2]{manolescu2023rasmussen}) to decide the sign for each summand of $\partial^e$, using the resolution choices for $D$. We omit the detail here as we are not going to use it in the paper. Then $KC_d^*(D)$ is defined as the chain complex \[\left(C(D),\partial\right)\coloneqq\left(\bigoplus_{v\in\{0,1\}^N}C(D_v),\sum_{e\colon u\to v}\partial^e\right).\]\begin{thm}[{\cite[Theorem 2.15]{manolescu2023rasmussen}}]
	The chain homotopy type of $KC_d^*(D)$ is a link invariant for $K$. 
\end{thm}

When $K$ is a class $1$ link, we can simplify the signs by using a preferred set of resolution choices as follows \cite[Theorem 2.14]{manolescu2023rasmussen}. Each resolution $D_v$ contains exactly one essential circle $C_v$ by Lemma \ref{lem reso circles}. An orientation of $C_v$ induces orientations on other circles in $D_v$ according to its `distance' from $C_v$. Fix an orientation of $K$ and an orientation for the essential circle in the oriented resolution of $D$. These data determine a set of resolution choices by requiring that all saddle cobordisms are orientable. Then \cite[Lemma 2.21]{manolescu2023rasmussen} ensures the that each $\partial^e$ sends generators in $C(D_u)$ to sums of generators in $C(D_v)$ of the same sign. To summarize, the complex $KC_d^*(D)$ is determined in this case by the local rules \cite[Table 6]{manolescu2023rasmussen} and a global sign assignment up to homotopy.

\begin{defn}\label{def Kh1}
	The Khovanov homology $\Kh(K)$ \cite{Gabrovek2013THECO} is defined as the homology of $KC_d(D)$ after specializing to $s=t=0$. The deformed invariant $\Kh_1(K)$ is defined as the homology of $KC_d(D)$ after specializing to $s=t=0$.
\end{defn}
 
 For a link $K$ with a marked point $p$, we can define a reduced version of the Khovanov homology $\Khr(K,p)$ following \cite{khovanov2003patterns}. For each resolution $D_v$, there is a unique circle containing $p$. The reduced Khovanov chain complex is obtained from the original one by replacing the factor $V$ (resp. $\overline{V}$ when the marked circle is essential) corresponding to the marked circle by $V/\langle X\rangle$ (resp. $\overline{V}/\langle \overline{X} \rangle$) at each vertex of the resolution cube.  When $K$ is a knot, the resulting homology group is independent to the position of $p$, and we may write it as $\Khr(K)$ for simplicity. 
 
The ordinary and the deformed versions of Khovanov homology satisfy a rank inequality, as follows.

\begin{prop}\label{prop Kh1 ss}
	For a link $K\subset \RP^3$, we have \[\rankk\Kh(K)\ge\rankk\Kh_1(K).\]Similar inequality holds for the reduced version.
\end{prop}

\begin{proof}
For the variant $\Kh_1(K)$, the differential decomposes as \[\partial^e=\partial_0^e+\partial_-^e.\]There is a third grading on $KC_d(D)$, denoted by $k$ in \cite{manolescu2023rasmussen}. The differential $\partial_0^e$ preserves the $k$ grading while $\partial_-^e$ decreases it by $2$. Hence, the filtration induced by the $k$ grading gives a spectral sequence from \[\Kh(K)=H^*(C(D),\partial_0)\] to $H^*(C(D),\partial)=\Kh_1(K)$, and the inequality follows from it.
\end{proof}

\subsection{Instanton Floer homology}

Kronheimer and Mrowka \cite{kronheimer2011khovanov} defined the \textit{singular instanton Floer homology} $\Is(Y,K)$ for a pair $(Y,K)$, where $Y$ is a smooth, closed, oriented $3$-manifold, and $K$ is an unoriented link in $Y$. Roughly speaking, it is constructed as a Morse homology of the (perturbed) Chern--Simons functional over a space of orbifold $\SO(3)$-connections on an associated orbifold bundle constructed from $(Y,K)$. Critical points of the unperturbed Chern--Simons functional (after quotienting by the gauge group) correspond to $\SU(2)$-representations of $\pi_1(Y\backslash K)$ that have trace zero around $K$ \textit{without} modulo conjugates. The resulting $\Is(Y,K)$ carries an absolute $\Z/4$ grading \cite[Proposition 4.4]{kronheimer2011khovanov}. When $K=\emptyset$, the \textit{framed instanton homology} $\Is(Y)\coloneqq\Is(Y,\emptyset)$ is further studied by Scaduto \cite{scaduto2015instantons}. There is also a reduced version $\In(Y,K)$ for link $K$ with a marked point. The unreduced version $\Is(Y,K)$ can be recovered as $\In(Y,K\sqcup U)$, where $U$ is an unknot near infinity with a marked point.

A cobordism \[(W,S)\colon (Y_1,K_1)\to (Y_0,K_0)\]induces a map \[\Is(W,S)\colon \Is(Y_1,K_1)\to \Is(Y_0,K_0)\]between $\Z/4$-graded abelian groups. By \cite[Proposition 4.4]{kronheimer2011khovanov}, the grading is shifted by \begin{equation}\label{eqn grading shift}
	-\chi(S)+b_0(\partial^+S)-b_0(\partial^-S)-\frac{3}{2}(\chi(W)+\sigma(W))+\frac{1}{2}(b_1(Y_0)-b_1(Y_1))
\end{equation} under $\Is(W,S)$. Similar result holds for the reduced version. In particular, when the knots are not present, the grading shift on framed instanton homology is given by\begin{equation}\label{eqn grading shift framed}
-\frac{3}{2}(\chi(W)+\sigma(W))+\frac{1}{2}(b_1(Y_0)-b_1(Y_1))
\end{equation}(cf. \cite[Proposition 7.1]{scaduto2015instantons}).

The map is only well-defined up to an overall sign. To resolve the sign ambiguity, we need to choose an \textit{$\In$-orientation} for $(W,S)$ in the sense of \cite[Definition 4.3]{kronheimer2011khovanov}. There is a canonical choice of $\In$-orientation for $(W,S)$ when $W$ is a product cobordism $[0,1]\times Y$ and $S$ is oriented.

As other Floer-theoretic invariants, singular instanton Floer homology satisfies several exact triangles. The first one we record here is a surgery exact triangle, which is essentially Floer's original exact triangle (cf. \cite{braam1995floer,scaduto2015instantons}) with an additional knot involved. The calculation in \cite[Section 5]{scaduto2015instantons} is local, and we can directly apply the argument there and replace the original instanton homology by singular instanton homology. A version for (sutured) instanton knot homology \cite{kronheimer2010knots} can be found in \cite[Theorem 2.19]{li2022instanton}.

\begin{thm}\label{thm surgery tri}
	Let $Y$ be a closed oriented $3$-manifold, $K\subset Y$ be a link, and $J\subset Y$ be a framed knot disjoint from $K$. Let $Y_r$ be the result of $r$-surgery along $J$, and $K_r\subset Y_r$ be the link corresponding to $K$. Then there is an exact triangle \[ \xymatrix@C=10pt@R=10pt{
		\Is(Y_0,K_0;\mu) \ar[rr]^{i} & & \Is(Y_1,K_1) \ar[dl]^{j} \\
		& \Is(Y,K)\ar[ul]^{k} & \\
	}. \]Here $\mu$ is a meridian of $J$, and $\Is(Y_0,K_0;\mu)$ is a version of singular instanton homology that only depends on $Y_0,\,K_0$, and the homology class of $\mu$ in $H_1(Y_0;\ZZ)$. See \cite[Section 7.5]{scaduto2015instantons} for details. The maps $i,j,k$ are induced by handle attaching cobordisms. A similar result holds for the reduced version.
\end{thm}

The second is an unoriented skein exact triangle \cite[Proposition 6.11]{kronheimer2011khovanov}. Assume that $K_0,K,K_1$ are links in $Y$ that only differ in a $3$-ball by the pattern showed in Figure \ref{01smoothing}. Then there are standard skein cobordisms that induce an exact triangle \[ \xymatrix@C=10pt@R=10pt{
	\In(Y,K_1) \ar[rr] & & \In(Y,K_0) \ar[dl] \\
	& \In(Y,K)\ar[ul] & \\
}{.}\]More generally, let $K$ be a link in a closed $3$-manifold $Y$, and let $B_1,B_2,\dots,B_N$ be $N$ disjoint $3$-balls in $Y$, intersecting $K$ as the middle of Figure \ref{01smoothing}. For each $v\in\{0,1\}^N$, we can form the link $K_v$ by changing $K$ in each $3$-ball locally to its $0$ or $1$ resolution according to $v$. For each pair $v\ge u$, there is a standard skein cobordism $S_{vu}$ from $K_v$ to $K_u$ that induces a map on instanton homology up to sign. Note that the direction here is opposite to the differentials of Khovanov homology.

\begin{prop}[{\cite[Corollary 6.9, 6.10]{kronheimer2011khovanov}}]\label{prop km ss}
 There is a spectral sequence with the $E_1$ term given by \[\bigoplus_{v\in\{0,1\}^N}\Is(Y,K_v)\]and converging to $\Is(Y,K)$. Moreover, if $\In$-orientations have been chosen for all cobordisms $S_{vu}$ so that the compatibility condition of \cite[Lemma 6.1]{kronheimer2011khovanov} holds, then the differential on the $E_1$ page is given by \[d_1=\sum_{i=1}^N\sum_{v-u=e_i}(-1)^{\delta(v,u)}\Is(S_{vu}).\]Here $e_i$ is the standard $i$-th basis vector in $\RR^N$, and $\delta(v,u)=\sum_{j=i}^Nv_j$.
\end{prop}

We do not have to worry about the sign ambiguity when using $\ZZ$ coefficients. In this case, the differential on the $E_1$ page is simply given by the sum of the maps $\Is(S_{vu})$.

\section{Instanton homology of crossingless links in $\RP^3$}\label{sec Is unlink}

The goal of this section is to describe the instanton Floer homology of crossingless links in $\RP^3$. Here, a \textit{crossingless link} means a link in $\RP^3$ that admits a diagram in $\RP^2$ without crossings. By Lemma \ref{lem reso circles}, every resolution of a class $0$ knot is a local unlink, while every resolution of a class $1$ knot  is a disjoint union of a local unlink and a projective unknot. These give all possible crossingless links. 

Let $U_n$ be a \textit{fixed} local unlink with $n$ components, and $U_n\coprod U$ be a \textit{fixed} crossingless link of class $1$ with $n+1$ components. We first establish isomorphisms between the instanton Floer homology for $U_n$ and $U_n\coprod U$ with certain graded abelian groups, and then investigate in what sense these isomorphisms are canonical under isotopies.

\subsection{The class $0$ unknot}

We begin with $\Is(\RP^3,U_0)=\Is(\RP^3)$. By \cite[Section 7.5]{scaduto2015instantons}, there is an exact triangle \[ \xymatrix@C=10pt@R=10pt{
	\Is(S^3) \ar[rr]^{i} & & \Is(\RP^3) \ar[dl]^{j} \\
	& \Is(S^3)\ar[ul]^{k} & \\
} \]associated to the $(1,2,\infty)$ surgery triad on the unknot in $S^3$. Recall that $\Is(S^3)$ is a free abelian group of rank $1$, supported in degree $0$. The map $k$ is induced by the cobordism $\mathbb{CP}^2\backslash(\operatorname{(int)}(D^4\sqcup -D^4))$ and has degree $1$ by the grading formula (\ref{eqn grading shift framed}). Therefore, $k$ is the zero map. Maps $i$ and $j$ are induced by the cobordism $DT^*S^2\backslash\operatorname{int}(D^4)$ and $-DTS^2\backslash\operatorname{int}(D^4)$ respectively, and they have degree $0$. Let $V_0$ be a rank $2$ free $\Z/4$-graded abelian group, supported in degree $0$. From the discussion above, we have \[\Is(\RP^3)\cong V_0\] as $\Z/4$-graded abelian groups. In particular, recall that a rational homology sphere is an \textit{instanton L-space} if \[\dim_\CC\Is(Y)=|H_1(Y;\Z)|.\] It follows that $\RP^3$ is an instanton L-space.

Let $U_1$ be the local unknot in $\RP^3$. By \cite[Corollary 5.9]{kronheimer2011khovanov}, there is an isomorphism\[ \Is(\RP^3,U_1)\cong \Is(S^3,U_1)\otimes\Is(\RP^3)=\Is(S^3,U_1)\otimes V_0,\]induced by an excision cobordism. Recall from \cite[Lemma 8.3]{kronheimer2011khovanov} that $\Is(S^3,U_1)$ can be canonically identified with a rank $2$ free $\Z/4$-graded abelian group $V$ with two distinguished generators $v_+$ and $v_-$ in degree $0$ and $-2$ respectively, characterized by a cap (resp. cup) cobordism $D_+$ (resp. $D_-$) from (resp. to) $(S^3,U_0)$.

\subsection{The class $1$ unknot}

To understand the instanton homology of the projective unknot $U'$, we exploit another surgery exact triangle. Let $J$ be the $1$-framed unknot in $S^3$, and $K$ be an unknot such that $J\cup K$ forms a Hopf link. Theorem \ref{thm surgery tri} gives an exact triangle \[ \xymatrix@C=10pt@R=10pt{
	\In(S^3,U_1) \ar[rr]^{i} & & \In(\RP^3,U') \ar[dl]^{j} \\
	& \In(S^3,U_1)\ar[ul]^{k} & \\
}. \]The map $k$ has degree $-2$ while $i,j$ have degree $0$ by the grading formula (\ref{eqn grading shift}). Recall that $\In(S^3,U_1)$ is a rank $1$ free abelian group supported in degree $0$. Therefore, $k$ is the zero map, and $\In(\RP^3,U')$ is a rank $2$ free abelian group supported in degree $0$. 

There is an unoriented skein exact triangle (cf. \cite[Lemma 8.3]{kronheimer2011khovanov}) relating the reduced and unreduced version of instanton homology groups:\begin{equation}\label{eqn skein U'}
\xymatrix@C=10pt@R=10pt{
		\In(\RP^3,U') \ar[rr]^{a} & & \In(\RP^3,U') \ar[dl]^{b} \\
		& \Is(\RP^3,U')\ar[ul]^{c} & \\
	}.
\end{equation}The map $a$ has degree $1$ by (\ref{eqn grading shift}) and hence is the zero map. Therefore, $\Is(\RP^3,U')$ is a rank $4$ free abelian group with two rank $2$ summands on grading $0$ and $-2$, respectively. This in turn implies the map $k$ is zero in the exact triangle  \begin{equation}\label{eqn Is class 1} \xymatrix@C=10pt@R10pt{
\Is(S^3,U_1) \ar[rr]^{i} & & \Is(\RP^3,U') \ar[dl]^{j} \\
& \Is(S^3,U_1)\ar[ul]^{k} & \\
}.\end{equation}Define $W=\Is(\RP^3,U')$. Let $w_\pm$ be the image of $v_\pm\in \Is(S^3,U_1)$ under $i$. The discussion above implies there is a (non-canonical) isomorphism \begin{equation}\label{eqn split U'}W\cong \Z\langle w_+,w_-\rangle\oplus V.
\end{equation} Denote the corresponding generators of the second summand by $w_\pm'$.

For further purpose, we interpret $\Is(\RP^3,U')$ in another way. Critical points of the unperturbed Chern--Simons functional for $(\RP^3,U')$ correspond to representations $\pi_1(\RP^3\backslash U')\to \SU(2)$ with $\operatorname{tr}a^2=0$, where $a$ is a generator of the fundamental group of $\RP^3\backslash U'\cong S^1\times D^2$. The representation variety contains two copies of $S^2$. The unperturbed Chern--Simons functional is Morse--Bott. After perturbation, the Floer chain complex is generated by four critical points of grading $0,0,-2,-2$ respectively, and there are no Floer differentials.

\subsection{Crossingless links and their isotopies}

For the fixed crossingless links $U_n$ and $U_n\coprod U$, we can compute the instanton homology by results from previous two subsections and a K\"unneth formula \cite[Corollary 5.9]{kronheimer2011khovanov}. We summarize the discussion in the following proposition.

\begin{prop}\label{prop E1 objects}
	We have isomorphisms of $\Z/4$-graded abelian groups\begin{align*}
		\Phi_n\colon V^{\otimes n}\otimes V_0&\to \Is(\RP^3,U_n),\\
		\Phi'_n\colon V^{\otimes n}\otimes W&\to\Is(\RP^3,U_n\sqcup U'),
	\end{align*}well-defined up to signs, and induced by excision cobordisms and  natural for split cobordisms from $U_n$ (resp. $U_n\coprod U'$) to itself (cf. \cite[Corollary 8.5]{kronheimer2011khovanov}). Here a \textnormal{split cobordism} is a cobordism of $U_n$ (resp. $U_n\coprod U'$) to itself in $\RP^3\times I$ which is a disjoint union of $n$ cobordisms of $U_1$ to $U_1$ (each contained in a standard ball $B^3\times I$) and a cobordism from $\RP^3$ (resp. $(\RP^3,U')$) to itself.
\end{prop}

To proceed further, we need to introduce the notion of singular instanton Floer homology with local coefficients \cite[Section 3.9]{kronheimer2011knot} (see also \cite[Section 3.1]{kronheimer2013gauge}). Roughly speaking, given a knot $K$ in a closed oriented $3$-manifold $Y$, we can define a local system $\Gamma$ on the configuration space $\cab(Y,K^\sharp)$ that encodes the holonomy along $K$ for orbifold connections on $Y$. The germs are free rank $1$ modules over $\car=\Z[t,t^{-1}]\subset \Z[\RR]$. Denote the corresponding homology group with local coefficients by $\Is(Y,K;\Gamma)$. Most of the argument in Section \ref{sec bg} remains true for Floer homology with local coefficients after mild modifications.

The advantage of using local coefficients here is that we can compare homotopic (not necessarily isotopic) cobordisms, which is the first step to ensure the identification in Proposition \ref{prop E1 objects} is canonical in the sense of \cite[Proposition 8.10]{kronheimer2011khovanov}. The following proposition is an analog of \cite[Proposition 5.15]{xie2021instantons}, which deals with the case of links in a thickened annulus.

\begin{prop}\label{prop hmtp inv}
	Let $K_0$ and $K_1$ be two crossingless links in $\RP^3$ in the same class, and let $S$ and $S'$ be two cobordisms in $\RP^3\times I$ from $K_0$ to $K_1$. If $S$ and $S'$ are homotopic relative to the boundary, then they induce the same map from $\Is(\RP^3,K_0)$ to $\Is(\RP^3,K_1)$.
\end{prop}

\begin{proof}[Proof]
	
	The proof is a formal adaption of  \cite[Proposition 5.15]{xie2021instantons}. The key observation is that the perturbed Chern--Simons functional for $(\RP^3,U')$ has four critical points, pairwise differing by an even degree, and hence there are no Floer differentials. Therefore,  $\Is(\RP^3,U';\Gamma)$ is a free $\car$-module of rank $4$. Similarly, $\Is(\RP^3,U_1;\Gamma)$ is also a free $\car$-module of rank $4$. Further, $\Is(\RP^3,K;\Gamma)$ is free for $K=U_n$ or $K=U_n\sqcup U'$ as \cite[Corollary 5.9]{kronheimer2011khovanov} still holds for Floer homology with local coefficients. 
	
	View $\Z$ as a $\car$-module $\car/(t-1)$. By the universal coefficient theorem, it suffices to show that $S$ and $S'$ induce the same map on Floer homology groups with local coefficients. Since $S$ and $S'$ are homotopic, they are related by a sequence of certain moves that introduce double points (see the proof of {\cite[Proposition 4.9]{xie2020instantons}}). By \cite[Proposition 5.2]{kronheimer2011knot}, we have \[(t^{-1}-t)^m\Is(\RP^3\times I,S;\Gamma)=(t^{-1}-t)^m\Is(\RP^3\times I,S';\Gamma),\]which completes the proof since  $\Is(\RP^3,K_0;\Gamma)$ and $\Is(\RP^3,K_1;\Gamma)$ are free $\car$-modules. 
\end{proof}

Given an isotopy $f_t\colon K\to \RP^3$, we can consider the trace cobordism \[S_f=\bigcup_{t\in I}\left(f_t(K)\times\{t\}\right)\subset \RP^3\times I.\] Proposition \ref{prop hmtp inv} implies that two isotopies from a crossingless link $K$ to itself induce the same map on instanton homology if they are homotopic relative to the boundary. The next proposition asserts that the induced maps are the same up to sign even if two isotopies are not homotopic.

\begin{prop}\label{prop diff hmtpy}
	Let $f$, $f'$ be two isotopies from a crossingless link $K$ to itself. Then the induced maps $\Is(S_f)$ and $\Is(S_{f'})$ are the same up to a sign.
\end{prop}

\begin{proof}
	By Proposition \ref{prop hmtp inv}, we can assume that the cobordisms $S_f$ and $S_f'$ are split, and it suffices to prove the assertion for the cases $K=U_1$ or $K=U'$. 
	
	The relative homotopy classes from $S^1\times I$ to $\RP^3\times I$ arising as traces of isotopies of $K$ are classified by $\pi_1(\RP^3)\cong \ZZ$. To see this, consider an isotopy $f_t\colon S^1\to\RP^3$. We can assume that there is a point $p$ such that $f_0(p)=f_1(p)$. The trace of $p$, given by $\cup(f_t(p)\times \{t\})$, is classified by $\pi_1(\RP^3)\cong\ZZ$ up to homotopy. Notice that $\pi_2(\RP^3\times I)$ is trivial, so there is a unique way to fill in the interior of the square $({S^1\backslash\{p\}})\times I$. Hence, the homotopy class of the trace is uniquely determined by the trace of $p$.
	
	When $K=U_1$, they induce the same map by the naturality under split cobordisms of the first isomorphism given in Proposition \ref{prop E1 objects}.
	
	Now let $K=U'$, and $S_1$ be the non-trivial cobordism from $U'$ to itself. Recall from Equation (\ref{eqn Is class 1}) that there is a short exact sequence \[0\to \Is(S^3,U_1)\xrightarrow{i}\Is(\RP^3,U')\xrightarrow{j}\Is(S^3,U_1)\to 0,\]where $i$ and $j$ are induced by handle attaching cobordisms. We have \[\Is(S_1)(w_+)=(\Is(S_1)\circ i)(v_+)=i(v_+)=w_+,\]and \[j(\Is(S_1)(w_+'))=\Is(S_1)(j(w_+'))=v_+.\]Moreover, $\Is(S_1)$ has order $2$ since $S_1\circ S_1$ is isotopic to the trivial cobordism. Therefore, $\Is(S_1)(w_+')=w_+'$. Similar argument applies for $w_-$ and $w_-'$. Hence, $\Is(S_1)$ is the identity map on $W$.
\end{proof}

The following proposition summarizes the discussion in Propositions \ref{prop hmtp inv} and \ref{prop diff hmtpy} (cf. \cite[Proposition 8.10]{kronheimer2011khovanov}).

\begin{prop}\label{prop E1 objs canonical}
	Let $\cau_n$ be a local unlink with $n$ components, and $\cau_N\coprod\cau'$ be a class $1$ crossingless link. Then there are canonical isomorphisms defined up to signs\begin{align*}
		\Psi_n\colon V^{\otimes n}\otimes V_0&\to \Is(\RP^3,\cau_n)\\
		\Psi_n'\colon V^{\otimes n}\otimes W&\to\Is(\RP^3,\cau_n\coprod \cau')
	\end{align*}given by composing the isomorphism from Proposition \ref{prop E1 objects} with $\Is(S)$. Here, $S$ is the trace of any isotopy from the standard $U_n$ (resp. $U_n\coprod U'$) to $\cau_n$ (resp. $\cau_n\coprod \cau'$).
\end{prop}

\section{Differentials on the $E_1$ page
}\label{sec E_1 differential}

In this section, we compute the differentials on the $E_1$ page of Kronheimer--Mrowka's spectral sequence, Proposition \ref{prop km ss}. We prove Theorems \ref{ss class 0} and \ref{ss class 1} in two subsections.

\subsection{The class $0$ case}

In the class $0$ case, the merging and splitting maps are similar to the $S^3$ case as no essential circle appears, and they have been calculated in \cite[Lemma 8.7]{kronheimer2011khovanov}. It remains to calculate the map induced by the $1$-to-$1$ bifurcation as showed in Figure \ref{fig 1-1bifur}. Write this cobordism as \[(\RP^3\times I,S)\colon (\RP^3,U_1)\to (\RP^3,U_1).\]If we omit the embedding, the surface $S$ is simply an $\RP^2$ with two disk removed. By the grading formula (\ref{eqn grading shift}), the induced map \[\Is(\RP^3\times I,S) \colon \Is(\RP^3,U_1)\to \Is(\RP^3,U_1)\]has degree $1$. Hence, it is the zero map since $\Is(\RP^3,U_1)$ is supported in even degrees. Combining with \cite[Lemma 8.7]{kronheimer2011khovanov} and the K\"unneth formula for split links \cite[Corollary 5.9]{kronheimer2011khovanov}, we obtain the following.

\begin{prop}\label{prop case 0 local cobordism}
	The merge map induced by a pair-of-pants cobordism from $U_1$ to $U_2$ is given by \[\Delta\otimes\idd_{V_0}\colon V\otimes V_0\to V\otimes V\otimes V_0,\]and the split map is given by \[\nabla\otimes\idd_{V_0}\colon V\otimes V\otimes V_0\to V\otimes V_0.\]Here $\Delta$ and $\nabla$ are defined as in \cite[Lemma 8.7]{kronheimer2011khovanov}. The 1-to-1 bifurcation induces the zero map on $V\otimes V_0$.
\end{prop}

\begin{proof}[Proof of Theorem \ref{ss class 0}]
	Fix a diagram for $K$ and consider a collection of crossingless links $K_v\,(v\in\{0,1\}^N)$ as above. By Proposition \ref{prop km ss}, there is a spectral sequence with the $E_1$ page given by \[\bigoplus_{v\in\{0,1\}^N}\Is(\RP^3,K_v)\]and converging to $\Is(\RP^3,K)$. It suffices to identify the $E_1$ page with the Khovanov chain complex. By Proposition \ref{prop E1 objs canonical}, we have a canonical identification induced by a cobordism map\[\Is(\RP^3,K_v)\cong V^{\otimes n_v}\otimes V_0\cong \operatorname{CKh}_{\bar{v}}(K)\otimes V_0=\operatorname{CKh}_{v}(m(K))\otimes V_0.\]Here $\bar{v}=(1,1,\dots,1)-v$, $n_v$ is the number of connected components in $K_v$, and $\operatorname{CKh}_{v}(m(K))$ is the chain group corresponding to the vertex $v$ in the complex that computes the Khovanov homology of $m(K)$. 
	
	We want to compare the differentials on the $E_1$ page and the differentials for $\operatorname{CKh}(m(K);\ZZ)$. When using $\ZZ$ coefficients, the former differential is given by sum of the maps induced by skein cobordisms without any sign ambiguity by Proposition \ref{prop km ss}. The skein cobordism $S_{vu}$ is a union of a product cobordism and either a pair-of-pants cobordism or a $1$-to-$1$ bifurcation. By Propositions \ref{prop hmtp inv} and \ref{prop diff hmtpy}, we can assume that $S_{vu}$ is split when computing its induced map on $\Is$. Therefore, by the naturality of the K\"unneth formula, $\Is(S_{vu})$ is a tensor product of the map in Proposition \ref{prop case 0 local cobordism} and the identity map on remaining components, which coincides with the differential on Khovanov chain complex of $m(K)$ using the same diagram. By the universal coefficient theorem, this implies the $E_2$ page of the spectral sequence is isomorphic to \[\Kh(m(K);\ZZ)\otimes_{\ZZ} (V_0\otimes_\Z \ZZ)=\Kh(m(K);\ZZ)\oplus\Kh(m(K);\ZZ).\]
	
	Now let $K$ be a link with one marked point $p$, and assume that $p$ is not a crossing in the diagram. The assertion about the reduced version follows from the same argument as \cite[Section 8.7]{kronheimer2011khovanov}. The $E_1$ page is obtained from the unreduced cube by replacing one $V$ ( corresponding to the circle containing $p$) by a factor $V/\langle v_-\rangle$, and the differentials are given by the quotient of corresponding maps in Proposition \ref{prop case 0 local cobordism}.
\end{proof}

\begin{rem}
	While the map induced by a cobordism between instanton Floer homology groups always has a global sign determined by the choice of the $\In$-orientation, the differentials in the Khovanov chain complex may have complicated signs depending on the set of auxiliary diagram choices. So it is not expected to obtain a spectral sequence over $\Z$ in the class $0$ case. Another evidence is that there is no checkerboard coloring (which is used in \cite[Section 8.1]{kronheimer2011khovanov} to find coherent orientations) for non-local diagrams in $\RP^2$.
\end{rem}

\subsection{The class $1$ case}

Now assume that $K\subset \RP^3$ is a link of class $1$. Four types of bifurcations may appear in the resolution cube of $K$: two of them involve merging and splitting for null-homologous circles, and their induced map are the same as in the case of $S^3$; the other two bifurcations involve merging (resp. splitting) from (resp. to) an essential circle and a null-homologous circle to (resp. from) an essential circle. See Figure \ref{fig essen bifur}.

\begin{figure}[hbtp]
	\centering
	\begin{tikzpicture}
		\draw[<->] (-0.55,0)--(0.55,0);
		
		
		\draw[thick] (-2,0.5) circle [radius=0.2];
		
		\draw[thick,green] (-2,0.3) to (-2,0);
				
		\draw[thick] (-3,0) to (-1,0);
		
		\draw[thick] (1,0) to (1.8,0) to [out=0,in=-90] (1.75,0.2) to[out=90,in=180] (2,0.45) to[out=0,in=90] (2.25,0.2) to[out=-90,in=180] (2.2,0) to (3,0);
		
		\draw[thick] (-2,0) circle [radius=1];

		\draw[thick] (2,0) circle [radius=1];
	\end{tikzpicture}
	\caption{A bifurcation on $\RP^2$ with an essential circle. Notice that the diameter in the middle is a closed curve as we are woring on $RP^2$, and it represents the projective unknot $U'$.}\label{fig essen bifur}
\end{figure}

Denote the pair-of-pants cobordism from $U_1\sqcup U'$ to $U'$ (resp. from $U'$ to $U_1\sqcup U'$) by $S_+$ (resp. $S_-$). The merging and splitting maps are given by \[\Is(S_+)\colon W\otimes V\to W\]and \[\Is(S_-)\colon W\to W\otimes V\]respectively.

\begin{prop}\label{prop case 1 local}
	There is a choice of splitting in (\ref{eqn split U'}) such that under the basis $w_+,w_-,w_+',w_-'$ of $W$, $\Is(S^+)$ and $\Is(S_-)$ are given by \begin{align*}
		w_\pm\otimes v_+\mapsto w_\pm&\quad w_\pm'\otimes v_+\mapsto w_\pm'\\
		w_+\otimes v_-\mapsto w_-&\quad w_-\otimes v_-\mapsto 0\\
		w_+'\otimes v_-\mapsto w_-'&\quad w_-'\otimes v_-\mapsto 0,
	\end{align*}and
\begin{align*}
	w_+&\mapsto w_+\otimes v_-+w_-\otimes v_+\\
	w_-&\mapsto v_-\otimes w_-\\
	w_+'&\mapsto w_+'\otimes v_-+w_-'\otimes v_+\\
	w_-'&\mapsto w_-'\otimes v_-
\end{align*}respectively.
\end{prop}

\begin{proof}
	We first choose a splitting in (\ref{eqn split U'}) in an arbitrary way, and denote the corresponding basis by $w_+,w_-,w_+',w_-'$. Let \[D_+\colon (S^3,U_0)\to (S^3,U_1)\]and \[D_-\colon (S^3,U_1)\to (S^3,U_0)\]be the cap and cup cobordism respectively.
	
	Recall from Section \ref{sec Is unlink} that there is an exact triangle
	
	 \[ \xymatrix@C=10pt@R10pt{
		\Is(S^3,U_1) \ar[rr]^{i} & & \Is(\RP^3,U') \ar[dl]^{j} \\
		& \Is(S^3,U_1)\ar[ul]^{k}. & \\
	}. \]Here $i,j,k$ are induced by cobordisms. Denote the cobordisms that induce $i,j$ by $P_+$ and $P_-$, respectively. Then $P_+$ is a product cobordism $(S^3,U_1)\times I$ with a $4$-dimensional $2$-handle attached along a $(-2)$-framed unknot in $S^3\times\{1\}$, and $P_-$ is the inverse of the product cobordism with a $4$-dimensional $2$-handle attached along a $2$-framed unknot. We have\begin{equation}\label{eqn proj surgery w+-}
	\Is(P_+)(v_\pm)=w_\pm,\,\Is(P_-)(w_\pm)=0,\,\Is(P_-)(w_\pm')=v_\pm.
\end{equation}

The cobordism $S_+\cup D_+$ is isotopic to the product cobordism from $(\RP^3,U')$ to itself. Hence, we have \[\Is(S_+)(w\otimes v_+)=\Is(S_+\cup D_+)(w)=w\]for all $w\in W$ by the functoriality of $\Is$.

From the functoriality of $\Is$, \[\Is(P_+)\circ \nabla=\Is(S_+)\circ \Is(P_+\#\idd).\]Here $P_+\#\idd$ is a cobordism from $(S^3,U_1)\#(S_3,U_1)$ to $(\RP^3,U')\#(S^3,U_1)$, formed as the connected sum of cobordisms $P_+$ and $(S^3,U_1)\times I$. By (\ref{eqn proj surgery w+-}) and \cite[Lemma 8.7]{kronheimer2011khovanov}, we have\begin{align*}
\Is(S_+)(w_+\otimes v_-)&=\Is(S_+)\circ \Is(P_+\#\idd)(v_+\otimes v_-)\\&=\Is(P_+)\circ \nabla(v_+\otimes v_-)\\
&=\Is(P_+)(v_-)=w_-.
\end{align*}A similar argument shows \[\Is(S_+)(w_-\otimes v_-)=0.\]

From the functoriality of $\Is$, \[\nabla\circ\Is(P_-\#\idd)= \Is(P_-)\circ \Is(S_+).\]By degree reason, we can assume \[\Is(S_+)(w_+'\otimes v_-)=\lambda_1w_-+\lambda_2w_-'.\]Then by (\ref{eqn proj surgery w+-}) and \cite[Lemma 8.7]{kronheimer2011khovanov},\begin{align*}
\lambda_2v_-&=\Is(P_-)(\lambda_1w_-+\lambda_2w_-')\\&=\Is(P_-)\circ\Is(S_+)(w_+'\otimes v_-)\\&=\nabla\circ\Is(P_-\#\idd)(w_+'\otimes v_-)
\\&=\nabla(v_+\otimes v_-)=v_-.
\end{align*}Therefore, $\lambda_2=1$. Replace $\lambda_1$ by $A$, and then \[\Is(S_+)(w_+'\otimes v_-)=Aw_-+w_-'.\] Similarly, there is some $B$ such that \[\Is(S_+)(w_-'\otimes v_-)=Bw_+.\]

The functoriality of $\Is$ implies \begin{align}\label{eqn func S_- 1}
	\Is(\idd\#D_-)\circ\Is(S_-)&=\idd\\\label{eqn func S_- 2}
\Is(S_-)\circ\Is(P_+)&=\Is(P_+\#\idd)\circ\Delta\\\label{eqn func S_- 3}
	\Is(P_-\#\idd)\circ\Is(S_-)&=\Delta\circ\Is(P_-).
\end{align} Dual argument applies for $\Is(S_-)$ to compute the splitting maps. For example, assume that \[\Is(S_-)(w_+)=\lambda_3w_+\otimes v_-+\lambda_4w_-\otimes v_++\lambda_5w_+'\otimes v_-+\lambda_6w_-'\otimes v_+.\]Plugging into (\ref{eqn func S_- 2}) gives $\lambda_3=\lambda_4=1$, $\lambda_5=\lambda_6=0$. Similarly, we have \begin{align*}
\Is(S_-)(w_-)&=w_-\otimes v_-\\
\Is(S_-)(w_+')&=w_+'\otimes v_-+w_-'\otimes v_++Cw_-\otimes v_+\\
\Is(S_-)(w_-')&=w_-'\otimes v_-+Dw_+\otimes v_+
\end{align*}for some $C$ and $D$. The following lemma claims that there is essentially only one undetermined variable.

\begin{lem}\label{lem sigma}
	We have $A=C$, and $B=D=0$. 
\end{lem}

Assuming Lemma \ref{lem sigma}, we can write $\widetilde{w}_+'=w_+'-Aw_+$. Then it can be verified that our new basis $w_+,w_-,\widetilde{w}_+',w_-'$ satisfies the condition. For example, we have \begin{align*}
\Is(S_-)(\widetilde{w}_+')&=w_+'\otimes v_-+w_-'\otimes v_+-Aw_+\otimes v_-\\
&=(w_+'-Aw_+)\otimes v_-+w_-'\otimes v_+\\
&=\widetilde{w}_+'\otimes v_-+w_-'\otimes v_+,
\end{align*}as expected.
\end{proof}

\begin{proof}[Proof of Lemma \ref{lem sigma}]
	We study an auxiliary operator $\sigma$ defined as in \cite[Section 8.3]{kronheimer2011khovanov}. For a link $K\subset \RP^3$ with a marked point $p$,  \[\sigma\colon\Is(\RP^3,K)\to\Is(\RP^3,K)\]is a degree $2$ operator given by the induced map of the cobordism \[\left(\RP^3\times[-1,1],(K\times [-1,1])\#_{p\times\{0\}}T^2\right).\]
	
	By the functoriality of $\Is$, we have $\sigma=\Is(S_+)\circ\Is(S_-)$. Therefore, for $K=U'$, we have\begin{align*}
		\sigma(w_+)&=2w_-\\
		\sigma(w_-)&=0\\
		\sigma(w_+')&=(A+C)w_-+2w_-'\\
		\sigma(w_-')&=(B+D)w_+.
	\end{align*}From the proof of \cite[Lemma 8.7]{kronheimer2011khovanov} and the K\"unneth formula, we know that for $K=U_1$, we have\begin{align*}
		\sigma(v_+)&=2v_-\\
		\sigma(v_-)&=0.
	\end{align*}
	
	By the functoriality again, we have \[\Is(S_+)\circ(\sigma\otimes\idd)=\Is(S_+)\circ(\idd\otimes\sigma)=\sigma\circ\Is(S_+).\]We know that\begin{align*}
		\Is(S_+)\circ(\sigma\otimes\idd)(w_+'\otimes v_+)&=(A+C)w_-+2w_-'\\
		\Is(S_+)\circ(\idd\otimes\sigma)(w_+'\otimes v_+)&=2Aw_-+2w_-'.
	\end{align*}This implies $A=C$. Similarly, we have \begin{align*}
		\Is(S_+)\circ(\sigma\otimes\idd)(w_+'\otimes v_-)&=2Bw_+\\
		\Is(S_+)\circ(\idd\otimes\sigma)(w_+'\otimes v_-)&=0\\
		\sigma\circ\Is(S_+)(w_+'\otimes v_-)&=(B+D)w_+.
	\end{align*}This implies $B=D=0$.
\end{proof}

\begin{proof}[Proof of Theorem \ref{ss class 1}]
	Fix a diagram for $K$ and consider a collection of crossingless links $K_v\,(v\in\{0,1\}^N)$ as above. The result follows from the same argument as in Theorem \ref{ss class 0}, but we replace Proposition \ref{prop case 0 local cobordism} by Proposition \ref{prop case 1 local}. For the reduced version, we also have the skein exact triangle (\ref{eqn skein U'}). The map $c$ in (\ref{eqn skein U'}) has degree zero, and is given by a quotient map \[c\colon W\to W/\langle w_-,w_-'\rangle.\]Denote \[V_r=V/\langle v_-\rangle,\,W_r=W/\langle w_-,w_-'\rangle.\]Depending on the position of the marked point,  the merging map $\In(S_+)$ has the form \[W_r\otimes V\to W_r\] or \[W\otimes V_r\to W_r.\]Moreover, it is given by the quotient map of $\Is(S_+)$. A similar argument holds for $\In(S_-)$. 
\end{proof}

\begin{rem}
For a class $1$ link, we can orient all the links $K_v$ and all the cobordisms $S_{vu}$ in a consistent way as we did for Khovanov homology in Section \ref{sec bg}. The orientations on $S_{vu}$ determine a set of choices of $\In$-orientations satisfying the condition of \cite[Lemma 6.1]{kronheimer2011khovanov}. It is hence expected to be able to lift the spectral sequence to $\Z$ coefficients if one can specify the signs in Proposition \ref{prop E1 objs canonical}. It seems that $\ZZ$ coefficients should be enough for most practical purposes, as the instanton knot homology is usually defined with $\CC$ coefficients.
\end{rem}

The rank inequality (\ref{eqn rank ineq KHI}) follows directly from the spectral sequences.

\begin{proof}[Proof of Corollary \ref{cor rank ineq}]
	The class $0$ case follows from Theorem \ref{ss class 0} and the universal coefficient theorem as \[2\dim\Kh(K;\ZZ)\ge \dim\Is(\RP^3,K;\ZZ)\ge \dim\Is(\RP^3,K;\CC).\]The class $1$ case similarly follows from Proposition \ref{prop Kh1 ss}, Theorem \ref{ss class 1} and the universal coefficient theorem:\begin{align*}
	2\dim\Kh(K;\ZZ)\ge 2\dim\Kh_1(K;\ZZ)\ge\\ \dim\Is(\RP^3,K;\ZZ)\ge \dim\Is(\RP^3,K;\CC).
	\end{align*}The reduced versions follow similarly from the spectral sequences.
\end{proof}

%
%

\section{Detection results}\label{sec detection}

We can now deduce Theorem \ref{main theorem} from the rank inequality (\ref{rank ineq}) and results in (sutured) instanton knot homology.

\begin{proof}[Proof of Theorem \ref{main theorem}]
	
	Assume first that $K$ is a class $0$ knot in $\RP^3$ with \[\Khr(K;\ZZ)\cong\ZZ.\] By the rank inequality (\ref{eqn rank ineq KHI}), we have \[\dim_{\CC}\KHI(\RP^3,K)\le 2.\]By \cite[Proposition 1.5]{li2022instanton}, we have \[\dim_{\CC}\KHI(\RP^3,K)\ge\dim_{\CC}\Is(\RP^3)=2.\]Therefore, we have \[\dim_{\CC}\KHI(\RP^3,K)=2.\] Now \cite[Theorem 1.5]{li2023enhanced} implies $K$ is the unknot since $\RP^3$ is an instanton L-space and $K$ is null-homologous.
	
	Now assume that $K$ is a class $1$ knot in $\RP^3$ with \[\Khr(K;\ZZ)\cong\ZZ.\] We still have \[\dim_{\CC}\KHI(\RP^3,K)\le 2.\]By \cite[Proposition 1.5]{li2022instanton}, we have \[\dim_{\CC}\KHI(\RP^3,K)=2. \]
	
	It remains to show that this rank condition implies $K$ is isotopic to $U'$, which essentially follows from \cite[Theorem 5.1]{baldwin2021small}. The proof there utilizes an Alexander grading on $\KHI(L,J)$ for an irreducible rational homology sphere $L$ and a primitive knot $J\subset L$ defined by choosing a rational Seifert surface $S$ as an admissible surface in the sense of \cite{ghosh2023decomposing}. While the admissibility condition doesn't hold directly in our case, we can apply a positive stabilization (see \cite[Definition 2.24]{li2022instanton}) on $S$ to make $S$ admissible, and then shift the induced grading by $-1/2$. By \cite[Lemma 2.25]{li2022instanton}, doing positive stabilization does not violate the tautness of the sutured manifold $(M_\pm,\gamma_\pm)$ obtained by decomposing $L\backslash\nu(J)$ along $S$. This modified Alexander grading still detects knot genus and whether a knot is fibered. The argument in \cite[Section 5]{baldwin2021small} 
	then remains valid.
\end{proof}

	We conclude this paper by some final remarks. First, one can try to apply the rank inequality (\ref{rank ineq}) to detect other knots and links. For example, it is interesting to classify all knots in $\RP^3$ with reduced Khovanov homology of rank $3$. However, at the moment, the author does not know how to complete the case studies when $\dim\KHI(\RP^3,K;\CC)=6$ similarly to \cite[Section 6.2]{li2023enhanced}. This motivates the following conjecture.
	
	\begin{conj}\label{conj khr3}
		Let $K$ be a null-homologous knot in $\RP^3$. If $\dim\Khr(K;\ZZ)=3$, then $K$ is either the local trefoil, or the knot $2_1$ in the knot table \cite{drobotukhina1994classification}.
	\end{conj}
	
	One can also ask if Khovanov homology can detect some links with small crossing numbers in $\RP^3$. One key ingredient in the Hopf link detection \cite{baldwin2019khovanov} is the module structure on Khovanov homology as discussed in \cite{xie2021earrings}. However, it is currently unclear to the author how to extract the module structure for class $0$ links, where the usage of $\Z$ coefficients is essential in \cite{xie2021earrings}. On the other hand, it is expected that results similar to \cite[Theorem 1.1]{xie2021earrings} can be proved for class $1$ links as working over $\Z$ is feasible.

	Another interesting and related question is whether one can use Khovanov homology or instanton Floer homology to detect local knots in $\RP^3$, or more generally, to detect the minimal linking number of a knot with a class $1$ unknot. In \cite{xie2019instanton}, Xie and Zhang showed that annular Khovanov homology detects local links in a solid torus. While Khovanov homology in $\RP^3$ is formally similar to the annular Khovanov homology, it has an obvious drawback that the $\Z$-valued $f$-grading is currently unavailable (cf. \cite[Proposition 2.17]{manolescu2023rasmussen}). Also, it seems difficult to define an additional grading on $\Is(\RP^3,K)$ using the $\mu^{\operatorname{orb}}$ action as done in \cite{xie2021instantons}. It is observed that for a non-local knot $K$, the rank of Khovanov homology of $K$ is heuristically larger than the rank of $\Is(K)$. 
	
	\begin{conj}\label{conj nonlocal}
		Let $K$ be a non-local null-homologous knot in $\RP^3$. Then \[2\dim\Khr(K;\ZZ)>\dimm\In(\RP^3,K;\CC).\]
	\end{conj}
	
	Conjecture \ref{conj nonlocal} would imply Conjecture \ref{conj khr3} as follows. If $K$ is local, then the Khovanov homology of $K$ coincides with the ordinary Khovanov homology for links in $S^3$, and \cite[Theorem 1.4]{baldwin2022khovanov} implies that $K$ is the trefoil. If $K$ is non-local, then \[\dim\KHI(\RP^3,K;\CC)\le 2\dim\Khr(K;\ZZ)-1=5.\]It is not hard to show that $\dim\KHI(\RP^3,K;\CC)$ is always an even integer (see, e.g. \cite[Theorem 1.6]{li2023instanton2}), and thus, $\KHI(\RP^3,K;\CC)$ must have dimension $4$ or $2$. If $\dim\KHI(\RP^3,K;\CC)=2$, \cite[Theorem 1.5]{li2023enhanced} implies $K$ is the unknot, which contradicts to the fact that $\rankk\Khr(K;\ZZ)=3$. If $\dim\KHI(\RP^3,K;\CC)=4$, $K$ is a genus-one-fibered knot by \cite[Theorem 1.8]{li2023enhanced}. By \cite[Theorem 1.2]{morimoto1989genus}, it must be one of the knot $2_1$ or $4_2$ in \cite{drobotukhina1994classification}. The latter knot has a larger Khovanov homology group, so $K$ must be $2_1$.
	
\bibliographystyle{hplain}
\bibliography{RP3}

\end{document}